\documentclass[a4paper,11pt,twoside,reqno]{amsart}
\usepackage[utf8]{inputenc}
\usepackage[T1]{fontenc}
\usepackage{amsmath,amssymb,amsthm,mathtools,mathrsfs}
\usepackage[margin=1in]{geometry}
\usepackage[plainpages=false,pdfpagelabels=true]{hyperref}
\usepackage{enumitem}
\newtheorem{theorem}{Theorem}[section]
\newtheorem{proposition}[theorem]{Proposition}

\newtheorem{corollary}[theorem]{Corollary}
\theoremstyle{definition}
\newtheorem{definition}[theorem]{Definition}
\newtheorem{example}[theorem]{Example}
\newtheorem{hypothesis}[theorem]{Hypothesis}
\theoremstyle{remark}
\newtheorem{remark}[theorem]{Remark}

\newenvironment{Thm}{\begin{theorem}}{\end{theorem}}
\newenvironment{Prop}{\begin{proposition}}{\end{proposition}}

\newenvironment{Bem}{\begin{remark}}{\end{remark}}

\newtheorem{problem}[theorem]{Problem}

\title{Response Geometry for Einstein Metrics}

\author{Anna Siffert}
\address{Universität M\"unster, Mathematisches Institut\\
Einsteinstr. 62\\
48149 M\"unster\\
Germany}
\email{asiffert@uni-muenster.de}

\begin{document}
\begin{abstract}
We develop a response geometry for smooth parameterized families of Einstein
metrics equipped with canonical probe data. The differential of the probe
observations defines a response bundle morphism; pulling back a fixed metric
on the observation bundle gives a positive-semidefinite response tensor whose
kernel is exactly the space of first-order invisible parameter directions.
On the complete-response locus this tensor is Riemannian and yields local
rigidity and quantitative reconstruction estimates whenever the response is
locally realized by an observation map.

We study rank-defect loci, the associated quotient geometry, Gram and
conditioning operators, response volume, and covariant variation formulas for
simple response eigenvalues. Under an explicit realizability hypothesis we
also show that a single scalar observation can be chosen to detect every
direction of a finite-dimensional parameter space. These constructions are
applied to the harmonic-probe package arising from the Einstein Detection
Principle. A finite-dimensional interval-matrix example illustrates the
quantitative formulas, while marked unit-volume flat two-tori provide a fully
explicit model: three harmonic-energy observations reconstruct the marked
metric globally, the induced response metric admits an explicit positively
curved hyperboloid realization, and its conditioning deteriorates toward the cusp.
\end{abstract}

\maketitle
%%%%%%%%%%%%%%%%%%%%%%%%%%%%%%%%%%%%%%%%%%%%%%%%%%%%%%%%%%%%%%%%%%%%%%%%%%%%%%%%
\section{Introduction}
\label{sec:introduction}
%%%%%%%%%%%%%%%%%%%%%%%%%%%%%%%%%%%%%%%%%%%%%%%%%%%%%%%%%%%%%%%%%%%%%%%%%%%%%%%%
%%%%%%%%%%%%%%%%%%%%%%%%%%%%%%%%%%%%%%%%%%%%%%%%%%%%%%%%%%%%%%%%%%%%%%%%%%%%%%%%
\subsection{From parameterized Einstein families to response geometry}
%%%%%%%%%%%%%%%%%%%%%%%%%%%%%%%%%%%%%%%%%%%%%%%%%%%%%%%%%%%%%%%%%%%%%%%%%%%%%%%%
This paper develops a differential geometry associated with finite-dimensional
probe observations of Einstein metrics. Rather than studying the ambient
class of all Einstein metrics, we work with a chosen smooth
finite-dimensional parameterized family of normalized Einstein metrics and
investigate the geometric structures induced on its parameter manifold by the
corresponding probe responses.

More precisely, let \(S\) be a smooth finite-dimensional manifold whose
points parameterize normalized Einstein metrics. No smooth structure is
assumed on the ambient class of Einstein metrics or on any corresponding
moduli space. Instead, all constructions are carried out on the parameter
manifold \(S\). This point of view encompasses local deformation families,
finite-dimensional moduli families whenever they exist, and many other smooth
parameterizations arising naturally in the study of Einstein metrics.

Throughout the paper we assume that the chosen parameterized family is
equipped with finite-dimensional response data satisfying the standing
hypotheses stated in the standing hypotheses introduced above. In particular, each
parameter \(p\in S\) is associated with a finite-dimensional observation
space
\(\mathcal Z_p\)
and a linear response operator
\(R_p:T_pS\longrightarrow\mathcal Z_p,\)
describing the first-order variation of the prescribed probe observations
under infinitesimal variations of the chosen parameterized family.

The analytical origin of these pointwise response operators is provided by
the Einstein Detection Principle developed in \cite{EDP}. The present paper
does not reprove that detection theory. Instead, it takes the resulting
pointwise response data as input and asks what differential geometry can be
constructed once the additional smooth bundle-packaging and fibre-metric
hypotheses stated below are imposed. In particular, smoothness of the
observation bundle and of the bundle morphism $R$ is an explicit hypothesis
of the present work and is not inferred from \cite{EDP}.

The central question addressed in this paper is therefore the following.
\begin{quote}
\emph{What differential-geometric structures, relative to a fixed probe
package, are naturally induced on the parameter manifold by the associated
response operators?}
\end{quote}
The viewpoint adopted here shifts the emphasis from detection to geometry.
Rather than viewing the response operators merely as tools for detecting or
distinguishing nearby Einstein metrics, we investigate the geometry that they
induce on the parameter manifold itself. As in many classical constructions
of differential geometry, the resulting tensorial structures can be studied
independently of the analytical procedure by which they were obtained.

The principal contribution of the paper is to establish the foundations of
this response geometry. Assuming that the pointwise response operators admit
a smooth bundle-theoretic organization, we introduce the associated response
bundle, define the response metric, develop the observable quotient geometry,
and establish corresponding notions of infinitesimal observability, response
rigidity and spectral observability.

%%%%%%%%%%%%%%%%%%%%%%%%%%%%%%%%%%%%%%%%%%%%%%%%%%%%%%%%%%%%%%%%%%%%%%%%%%%%%%%%
\subsection{Relation to established observability geometries}
%%%%%%%%%%%%%%%%%%%%%%%%%%%%%%%%%%%%%%%%%%%%%%%%%%%%%%%%%%%%%%%%%%%%%%%%%%%%%%%%
The elementary mechanism behind a response metric---pulling back a metric from
an observation space by a parameter-to-observation map---is classical and is
not claimed as new here.  In statistics and information geometry, Fisher-type
metrics measure local distinguishability of parameterized probability models;
in nonlinear least squares, the manifold of model predictions carries the
metric induced from data space; and in systems and control, observability
Gramians quantify sensitivity of measured outputs to state perturbations; see,
for example, \cite{AmariNagaoka,BatesWatts,TranstrumMachtaSethna,Vaidya}.
These theories provide important conceptual precedents for using a positive
semidefinite Gram tensor and its spectrum to encode observability and
conditioning.

The contribution of the present paper is therefore deliberately narrower.  We
specialize this general geometric idea to smooth parameterized families of
normalized Einstein metrics equipped with the finite-dimensional probe data of
\cite{EDP}; formulate the resulting response bundle and observable quotient
geometry without assuming that an ambient Einstein moduli space is smooth;
and develop the corresponding constant-rank, spectral, covariant-variation,
and scalar-probe constructions in that setting.  The minimal scalar probe theorem is a finite-dimensional consequence of the
stated functional-analytic surjectivity hypothesis, rather than a claim of a
new general measurement principle.  Likewise, the Cauchy--Binet and singular-value formulas used later
are standard linear algebra.  The marked flat-torus section is included as an
exact model of this Einstein response framework: the classical geometry of
marked flat tori is not new, while the explicitly chosen three-channel
response metric, its reconstruction formula, conditioning analysis, and
hyperboloid realization are the conclusions specific to that observation
system.

%%%%%%%%%%%%%%%%%%%%%%%%%%%%%%%%%%%%%%%%%%%%%%%%%%%%%%%%%%%%%%%%%%%%%%%%%%%%%%%%
\subsection{The response metric}
%%%%%%%%%%%%%%%%%%%%%%%%%%%%%%%%%%%%%%%%%%%%%%%%%%%%%%%%%%%%%%%%%%%%%%%%%%%%%%%%
The central geometric object of the theory is the \emph{response metric}.
Under the standing hypotheses, the pointwise response operators assemble into
a smooth vector bundle morphism
$$R:TS\longrightarrow\mathcal Z,$$
where \(\mathcal Z\) denotes the observation bundle. After fixing a smooth
positive-definite fibre metric
\(h^{\mathcal Z},\)
on the observation bundle, we define the response metric by the pullback
formula
$$g_R=R^{*}h^{\mathcal Z}.$$
Equivalently,
$$(g_R)_p(X,Y) = h^{\mathcal Z}_p(R_pX,R_pY), \qquad X,Y\in T_pS.$$
Thus \(g_R\) is simply the Gram tensor determined by the response vectors of
the chosen probe package.

Geometrically, the response metric measures first-order distinguishability
through the prescribed probe observations. Its kernel consists precisely of
those infinitesimal parameter variations that are invisible to the chosen
probe package. Consequently, \(g_R\) is generally positive semidefinite
rather than positive definite. On regions where the response bundle morphism
has constant rank it descends canonically to a Riemannian metric on the
observable quotient bundle, while on the complete-response locus it becomes
Riemannian itself.

The response metric serves as the organizing object of the present paper.
Observable quotient geometry, local rigidity, spectral observability,
covariant variation formulas and minimal scalar probe families all arise
naturally from this single construction.
%%%%%%%%%%%%%%%%%%%%%%%%%%%%%%%%%%%%%%%%%%%%%%%%%%%%%%%%%%%%%%%%%%%%%%%%%%%%%%%%
\subsection{Main results}

The construction has four principal consequences. First, the response tensor
is precisely the Gram form of the response operator, so its null directions
are the infinitesimally invisible parameter variations and it is Riemannian
exactly on the complete-response locus. Second, when the response is locally
the differential of an observation map, complete response implies local
identifiability and a local Lipschitz reconstruction estimate. Third, the Gram
operator packages quantitative observability through its spectrum,
conditioning, determinant, and covariant variation. Fourth, under the stated
realizability hypothesis, finite-dimensional complete response can be
compressed to one scalar probe without losing first-order detection.

The general formalism is complemented by two explicit illustrations. A
validated interval-matrix calculation exhibits the finite-dimensional
spectral estimates, and the marked flat-torus model gives exact global
reconstruction and an explicit complete response metric with non-uniform
conditioning.

\subsection{Parameterized Einstein families}
%%%%%%%%%%%%%%%%%%%%%%%%%%%%%%%%%%%%%%%%%%%%%%%%%%%%%%%%%%%%%%%%%%%%%%%%%%%%%%%%
Throughout the paper let \(M\) be a fixed smooth manifold.

The basic object of the present paper is not the ambient space of all
Einstein metrics but a chosen smooth family of normalized Einstein metrics.
\begin{definition}[Parameterized Einstein family]
A \emph{parameterized Einstein family} on \(M\) consists of
\begin{enumerate}
\item
a smooth finite-dimensional manifold \(S\);
\item
a family of Riemannian metrics
$\{g_p\}_{p\in S}$
on \(M\), depending smoothly on the parameter \(p\), such that every
\(g_p\) is a normalized Einstein metric.
\end{enumerate}
\end{definition}

Here smooth dependence means that, in local coordinates on
\(S\times M\), the coefficients of the metric tensor depend smoothly on
both the parameter and the spatial variables.

No smooth structure is assumed on the ambient class of normalized Einstein
metrics. In particular, the parameter manifold \(S\) is not required to be a
local moduli space, nor to contain all nearby Einstein metrics. It is simply
part of the chosen data.
Consequently, every tangent vector
$X\in T_pS$
represents an infinitesimal variation of the chosen parameterized Einstein
family. No assertion is made that every infinitesimal Einstein deformation
arises in this way.

\begin{remark}
This framework includes local deformation families, finite-dimensional
moduli families whenever these exist, and any other smooth
finite-dimensional parameterization of normalized Einstein metrics.

In cohomogeneity-one Einstein problems one often first constructs a smooth
parameter space of local Einstein germs or propagated local solutions.
Whenever a smooth subfamily of globally defined normalized Einstein metrics
is obtained, the present theory applies directly to that family.
\end{remark}
%%%%%%%%%%%%%%%%%%%%%%%%%%%%%%%%%%%%%%%%%%%%%%%%%%%%%%%%%%%%%%%%%%%%%%%%%%%%%%%%
\subsection{Pointwise response data}
%%%%%%%%%%%%%%%%%%%%%%%%%%%%%%%%%%%%%%%%%%%%%%%%%%%%%%%%%%%%%%%%%%%%%%%%%%%%%%%%
The Einstein Detection Principle developed in \cite{EDP} associates finite-dimensional
response data with parameterized Einstein families satisfying its hypotheses.

Accordingly, throughout the present paper we assume that the chosen
parameterized Einstein family is equipped with pointwise response operators
\(R_p:T_pS\longrightarrow\mathcal Z_p,\)
where \(\mathcal Z_p\) denotes a finite-dimensional observation space.
These pointwise response operators constitute the analytical input for the
response geometry developed below.

The results of \cite{EDP} establish the existence of the pointwise response
operators under the hypotheses stated there. It does not assert that the
observation spaces assemble into a vector bundle or that the response
operators vary smoothly with the parameter. Those additional global
assumptions are imposed explicitly in the present paper.
%%%%%%%%%%%%%%%%%%%%%%%%%%%%%%%%%%%%%%%%%%%%%%%%%%%%%%%%%%%%%%%%%%%%%%%%%%%%%%%%
\subsection{Standing hypotheses}
%%%%%%%%%%%%%%%%%%%%%%%%%%%%%%%%%%%%%%%%%%%%%%%%%%%%%%%%%%%%%%%%%%%%%%%%%%%%%%%%
Throughout the remainder of the paper we assume the following.
\begin{enumerate}
\item
The observation spaces have locally constant dimension and assemble into a
smooth vector bundle
$\pi_{\mathcal Z}:\mathcal Z\longrightarrow S.$
\item
The observation bundle carries a smooth positive-definite fibre metric
$h^{\mathcal Z}.$
\item
The pointwise response operators vary smoothly with the parameter and
assemble into a smooth vector bundle morphism
$
R:TS\longrightarrow\mathcal Z.$
\item
Whenever equivariance under isometries is required, the chosen
parameterized Einstein family, the observation bundle and the response
bundle morphism are assumed to be compatible with the corresponding
symmetry.
\item
Whenever kernel bundles, image bundles or quotient bundles are used,
the corresponding constant-rank hypotheses are imposed explicitly.
\end{enumerate}
Unless stated otherwise, all manifolds, vector bundles, bundle morphisms,
fibre metrics and connections are assumed to be smooth.

The response geometry developed in this paper is therefore attached to the
chosen parameterized Einstein family together with the prescribed probe
package and observation metric. It is not an intrinsic invariant of the
Einstein equations alone, but of these geometric data.

\subsection*{Acknowledgements}
The author made extensive use of Mathematica and Sage.

\smallskip

AI is evolving at a rapid pace. 
When I started this project, AI was primarily an editing tool and useful for literature searches. Now it is capable of doing so much more:
The author conceived and developed the mathematical direction of the project, formulated the principal questions and connections investigated in the paper, and developed the mathematical arguments and proofs. OpenAI's ChatGPT was used extensively as an interactive mathematical assistant to develop and refine proof arguments, explore alternative approaches, test potential counterexamples, identify hidden assumptions and possible gaps, perform repeated mathematical consistency audits, and assist with the organization and exposition of the manuscript. The author reviewed the resulting arguments and takes full responsibility for the mathematical content and conclusions of the paper.

%%%%%%%%%%%%%%%%%%%%%%%%%%%%%%%%%%%%%%%%%%%%%%%%%%%%%%%%%%%%%%%%%%%%%%%%%%%%%%%%
\section{The Response Bundle}
\label{sec:response-bundle}
%%%%%%%%%%%%%%%%%%%%%%%%%%%%%%%%%%%%%%%%%%%%%%%%%%%%%%%%%%%%%%%%%%%%%%%%%%%%%%%%
Throughout this section let
\((S,\{g_p\}_{p\in S})\)
be a parameterized Einstein family satisfying the standing hypotheses introduced
above. The response bundle morphism provides the basic geometric object from
which all subsequent constructions are derived.

The primary object of the theory is the smooth vector bundle morphism
\(R:TS\to\mathcal Z\). Using the canonical identification
\(\operatorname{Hom}(TS,\mathcal Z)\cong T^*S\otimes\mathcal Z\), this
morphism may equivalently be regarded as a smooth section of the response
bundle \(\operatorname{Hom}(TS,\mathcal Z)\).
We shall freely pass between these two equivalent viewpoints.
No global observation map is assumed. Whenever there exists a smooth
observation map whose differential identifies with the response bundle
morphism, we simply write
\(R=d\Theta.\)
The precise hypotheses under which such an observation map exists are stated
later when local rigidity is discussed.

%%%%%%%%%%%%%%%%%%%%%%%%%%%%%%%%%%%%%%%%%%%%%%%%%%%%%%%%%%%%%%%%%%%%%%
\subsection{The response bundle}
%%%%%%%%%%%%%%%%%%%%%%%%%%%%%%%%%%%%%%%%%%%%%%%%%%%%%%%%%%%%%%%%%%%%%%%%%%%%%%%%
Let
\(\pi_{\mathcal Z}:\mathcal Z\longrightarrow S\)
denote the observation bundle.
\begin{definition}
The \emph{response bundle} is the vector bundle
\[
\mathscr R
:=
\operatorname{Hom}(TS,\mathcal Z).
\]
\end{definition}
Its fibre over \(p\in S\) is
\(\mathscr R_p = \operatorname{Hom}(T_pS,\mathcal Z_p).\)
Since
\(\operatorname{Hom}(V,W) \cong V^*\otimes W,\)
there is a canonical vector bundle isomorphism
$$\mathscr R \cong T^*S\otimes\mathcal Z.$$
In particular, the response bundle has finite rank given by
\((\dim S)(\operatorname{rank}\mathcal Z).\)
The response bundle morphism
\(R:TS\longrightarrow\mathcal Z\)
therefore determines canonically a smooth section
\(R\in \Gamma(\mathscr R),\)
which we call the \emph{response section}.
%%%%%%%%%%%%%%%%%%%%%%%%%%%%%%%%%%%%%%%%%%%%%%%%%%%%%%%%%%%%%%%%%%%%%%%%%%%%%%%%
\subsection{Smoothness and naturality}
%%%%%%%%%%%%%%%%%%%%%%%%%%%%%%%%%%%%%%%%%%%%%%%%%%%%%%%%%%%%%%%%%%%%%%%%%%%%%%%%
By the standing hypotheses, the pointwise response operators vary smoothly
with the parameter and assemble into the bundle morphism
\(R:TS\longrightarrow\mathcal Z.\)
Equivalently,
\(R\in\Gamma(\mathscr R).\)
This smooth bundle packaging is an explicit assumption of the present paper;
it is not asserted by the results of \cite{EDP} alone.
Whenever equivariance under isometries is required, it is assumed that the
chosen parameterized Einstein family, the observation bundle and the
response bundle morphism are compatible with the corresponding symmetry.
Every construction developed below depends only on these objects together
with fibrewise linear algebra and is therefore natural with respect to the
assumed symmetry.
%%%%%%%%%%%%%%%%%%%%%%%%%%%%%%%%%%%%%%%%%%%%%%%%%%%%%%%%%%%%%%%%%%%%%%%%%%%%%%%%
\subsection{Invisible and observable bundles}
%%%%%%%%%%%%%%%%%%%%%%%%%%%%%%%%%%%%%%%%%%%%%%%%%%%%%%%%%%%%%%%%%%%%%%%%%%%%%%%%
Assume that the response bundle morphism
\(R:TS\longrightarrow\mathcal Z\)
has constant rank on an open subset
\(U\subset S.\)
Its kernel consists of infinitesimal parameter variations that are invisible
to the chosen probe package, whereas its image consists of all realizable
first-order observations.

\begin{definition}
The vector subbundle
\(\mathcal K = \ker R \subset TS|_U\)
is called the \emph{invisible bundle}.
The vector subbundle
\(\mathcal I = \operatorname{Im}R \subset \mathcal Z|_U\)
is called the \emph{observable bundle}.
\end{definition}

\begin{proposition}
\label{prop:observable-bundles}
If the response bundle morphism has constant rank \(r\), then
\begin{enumerate}
\item
\(\mathcal K\) and \(\mathcal I\) are smooth vector bundles;
\item
\[
\operatorname{rank}\mathcal K
=
\dim S-r,
\qquad
\operatorname{rank}\mathcal I
=
r;
\]
\item
the response bundle morphism induces a canonical vector bundle isomorphism
\(TS|_U/\mathcal K \cong \mathcal I.\)
\end{enumerate}
\end{proposition}
\begin{proof}
The statement is the standard constant-rank theorem for smooth vector bundle
morphisms. The quotient identification is the fibrewise first isomorphism
theorem.
\end{proof}
%%%%%%%%%%%%%%%%%%%%%%%%%%%%%%%%%%%%%%%%%%%%%%%%%%%%%%%%%%%%%%%%%%%%%%%%%%%%%%%%
\subsection{The complete-response locus}
%%%%%%%%%%%%%%%%%%%%%%%%%%%%%%%%%%%%%%%%%%%%%%%%%%%%%%%%%%%%%%%%%%%%%%%%%%%%%%%%
The response bundle morphism is said to be \emph{complete} at a parameter if
it detects every infinitesimal variation of the chosen parameterized
Einstein family.
\begin{definition}
The \emph{complete-response locus} is
\[
S_{\mathrm{comp}}
=
\{
p\in S:
\ker R_p=\{0\}
\}.
\]
Its complement
\(S_{\mathrm{deg}} = S\setminus S_{\mathrm{comp}}\)
is called the \emph{response-degeneracy locus}.
\end{definition}
\begin{proposition}
\label{prop:openness-complete-response-locus}
The complete-response locus is open.
\end{proposition}
\begin{proof}
Choose local trivializations of \(TS\) and \(\mathcal Z\) near a point
\(p_0\in S_{\mathrm{comp}}\). The bundle morphism is represented by a smooth
matrix-valued map. Since \(R_{p_0}\) is injective, one
\((\dim S)\times(\dim S)\) minor is non-zero. By continuity this minor
remains non-zero nearby, so injectivity persists.
\end{proof}
On the complete-response locus the invisible bundle vanishes.
Consequently,
\(TS \cong \mathcal I,\)
and
\[
S_{\mathrm{comp}}
=
\{
p\in S:
\operatorname{rank}R_p=\dim S
\}.
\]
Hence complete response is possible only if
\[
\operatorname{rank}\mathcal Z
\ge
\dim S.
\]
%%%%%%%%%%%%%%%%%%%%%%%%%%%%%%%%%%%%%%%%%%%%%%%%%%%%%%%%%%%%%%%%%%%%%%%%%%%%%%%%
\subsection{Analytic rank-defect loci}
%%%%%%%%%%%%%%%%%%%%%%%%%%%%%%%%%%%%%%%%%%%%%%%%%%%%%%%%%%%%%%%%%%%%%%%%%%%%%%%%
Assume that the parameter manifold, the observation bundle and the response
bundle morphism are real analytic.
For
\[
0\le
k
\le
\min
\{
\dim S,
\operatorname{rank}\mathcal Z
\},
\]
define
\[
\mathcal D_k(R)
=
\{
p\in S:
\operatorname{rank}R_p\le k
\}.
\]
\begin{proposition}
Each
\(\mathcal D_k(R)\)
is a closed real-analytic subset of \(S\). Locally it is defined by the
vanishing of the \((k+1)\times(k+1)\) minors of a matrix representative of
the response bundle morphism.
\end{proposition}
\begin{proof}
In analytic local trivializations the entries of the response matrix are
analytic functions. The statement therefore follows from the standard
characterization of matrix rank by minors together with the local theory of
real-analytic sets; see, for example, \cite{KrantzParks}.
\end{proof}
Without additional transversality assumptions the rank-defect loci may have
singularities and need not be smooth submanifolds.
%%%%%%%%%%%%%%%%%%%%%%%%%%%%%%%%%%%%%%%%%%%%%%%%%%%%%%%%%%%%%%%%%%%%%%%%%%%%%%%%
\section{The Response Metric}
\label{sec:response-metric}
%%%%%%%%%%%%%%%%%%%%%%%%%%%%%%%%%%%%%%%%%%%%%%%%%%%%%%%%%%%%%%%%%%%%%%%%%%%%%%%%
The response bundle morphism
\(R:TS\to\mathcal Z\)
distinguishes invisible directions \(\ker R\) from observable first-order
variations. Once the observation bundle is equipped with its fixed
positive-definite fibre metric \(h^{\mathcal Z}\), this information pulls back
to the parameter manifold. The resulting geometry is relative to the chosen
probe package and observation metric; it is not an intrinsic metric of the
Einstein equations alone.

Throughout this section, \((S,\{g_p\}_{p\in S})\) satisfies the standing
hypotheses of the standing hypotheses introduced above.
%%%%%%%%%%%%%%%%%%%%%%%%%%%%%%%%%%%%%%%%%%%%%%%%%%%%%%%%%%%%%%%%%%%%%%%%%%%%%%%%
\subsection{Definition}
%%%%%%%%%%%%%%%%%%%%%%%%%%%%%%%%%%%%%%%%%%%%%%%%%%%%%%%%%%%%%%%%%%%%%%%%%%%%%%%%
\begin{definition}[Response metric]
\label{def:response-metric}
The \emph{response metric} is the smooth symmetric covariant \(2\)-tensor
\(g_R=R^*h^{\mathcal Z}\), i.e.
\((g_R)_p(X,Y)=h^{\mathcal Z}_p(R_pX,R_pY)\).
Equivalently, \((g_R)_p\) is the Gram form of
\(R_p:T_pS\to\mathcal Z_p\).
\end{definition}
%%%%%%%%%%%%%%%%%%%%%%%%%%%%%%%%%%%%%%%%%%%%%%%%%%%%%%%%%%%%%%%%%%%%%%%%%%%%%%%%
\subsection{Elementary properties}
%%%%%%%%%%%%%%%%%%%%%%%%%%%%%%%%%%%%%%%%%%%%%%%%%%%%%%%%%%%%%%%%%%%%%%%%%%%%%%%%
The response metric is the fundamental geometric object associated with the
chosen response bundle morphism.
\begin{proposition}
\label{prop:elementary-response-metric-properties}
The response metric is a smooth symmetric positive-semidefinite covariant
\(2\)-tensor.
Moreover,
\(\ker(g_R)_p = \ker R_p,\)
for every parameter \(p\in S\).
\end{proposition}
\begin{proof}
Smoothness and symmetry follow from those of \(R\) and
\(h^{\mathcal Z}\). Moreover
\[
(g_R)_p(X,X)=\|R_pX\|_{h^{\mathcal Z}}^2\ge0,
\]
so \(g_R\) is positive semidefinite. Since \(h^{\mathcal Z}\) is positive
definite, \((g_R)_p(X,Y)=0\) for every \(Y\) if and only if \(R_pX=0\).
Hence \(\ker(g_R)_p=\ker R_p\).
\end{proof}
The kernel identity shows that the response metric measures infinitesimal
parameter variations only through their observable first-order responses.
Invisible parameter variations are precisely the null directions of the
response metric.
%%%%%%%%%%%%%%%%%%%%%%%%%%%%%%%%%%%%%%%%%%%%%%%%%%%%%%%%%%%%%%%%%%%%%%%%%%%%%%%%
\subsection{Complete-response geometry}
%%%%%%%%%%%%%%%%%%%%%%%%%%%%%%%%%%%%%%%%%%%%%%%%%%%%%%%%%%%%%%%%%%%%%%%%%%%%%%%%
Recall that
\(S_{\mathrm{comp}}=\{p\in S:R_p\text{ is injective}\}\), an open subset of
\(S\).

\begin{theorem}
\label{thm:complete-response-riemannian-metric}
The restriction
\(g_R|_{S_{\mathrm{comp}}}\)
is a smooth Riemannian metric.

Equivalently, for every parameter
\(p\),
the following are equivalent:
\begin{enumerate}
\item
\(p\in S_{\mathrm{comp}}\);
\item
\(R_p\) is injective;
\item
\((g_R)_p\) is positive definite.
\end{enumerate}
\end{theorem}
\begin{proof}
By Proposition~\ref{prop:elementary-response-metric-properties},
\((g_R)_p(X,X)=\|R_pX\|^2\). Thus \((g_R)_p\) is positive definite exactly
when \(R_p\) is injective; smoothness is already established.
\end{proof}

Thus the response metric becomes an honest Riemannian metric precisely when
the chosen probe package detects every infinitesimal variation of the
parameterized Einstein family.
%%%%%%%%%%%%%%%%%%%%%%%%%%%%%%%%%%%%%%%%%%%%%%%%%%%%%%%%%%%%%%%%%%%%%%%%%%%%%%%%
\subsection{Observable quotient geometry}
%%%%%%%%%%%%%%%%%%%%%%%%%%%%%%%%%%%%%%%%%%%%%%%%%%%%%%%%%%%%%%%%%%%%%%%%%%%%%%%%
Suppose the response bundle morphism has constant rank on an open subset
\(U\subset S.\)
Then the invisible bundle
\(\mathcal K = \ker R\)
and observable bundle
\(\mathcal I = \operatorname{Im}R\)
are smooth vector bundles.

The response metric therefore descends canonically to the observable
quotient.
\begin{proposition}
\label{prop:observable-quotient-metric}
There exists a unique smooth positive-definite fibre metric
\(\overline g_R\)
on
\(TS|_U/\mathcal K\)
such that
\(q^{*}\overline g_R = g_R,\)
where
\(q:TS|_U \longrightarrow TS|_U/\mathcal K\)
is the quotient projection.

Moreover,
the induced bundle isomorphism
\(\overline R: TS|_U/\mathcal K \longrightarrow \mathcal I\)
is an isometry with respect to the restricted observation metric.
\end{proposition}
%%%%%%%%%%%%%%%%%%%%%%%%%%%%%%%%%%%%%%%%%%%%%%%%%%%%%%%%%%%%%%%%%%%%%%%%%%%%%%%%
\subsection{Local rigidity and reconstruction}
%%%%%%%%%%%%%%%%%%%%%%%%%%%%%%%%%%%%%%%%%%%%%%%%%%%%%%%%%%%%%%%%%%%%%%%%%%%%%%%%
The response bundle morphism need not arise from a global observation map.
When it does, positivity of the response metric has strong geometric
consequences.
Suppose \(U\subset S\) is open, \((N,h^N)\) is Riemannian,
\(\Theta:U\to N\) is smooth, and there is a vector-bundle isomorphism
\(\Psi:\mathcal Z|_U\to\Theta^*TN\) satisfying \(\Psi\circ R=d\Theta\).
No compatibility between \(h^{\mathcal Z}\) and \(h^N\) is needed below:
fibrewise invertibility of \(\Psi\) gives
\(\ker d\Theta_q=\ker R_q\). If \(\Psi\) is additionally fibrewise
isometric, then \(g_R=\Theta^*h^N\).
\begin{theorem}[Local response rigidity]
\label{thm:response-rigidity}
Suppose that
\((g_R)_p\)
is positive definite at some parameter
\(p\in U.\)
Then there exists a neighbourhood
\(U_0\subset U\)
of \(p\) on which
\(\Theta\)
is a smooth embedding.
In particular, nearby members of the chosen parameterized Einstein family are
uniquely determined by their observation data.
\end{theorem}
\begin{proof}
Positive definiteness of \((g_R)_p\) is equivalent to injectivity of
\(R_p\). Since
\[
d\Theta_p=\Psi_p\circ R_p
\]
and \(\Psi_p\) is an isomorphism, \(d\Theta_p\) is injective. By the local
immersion theorem, after shrinking to a neighbourhood of \(p\), the map
\(\Theta\) is a smooth embedding onto an embedded submanifold of \(N\).
\end{proof}
The preceding theorem is entirely local. It establishes uniqueness only
within the chosen parameterized Einstein family.
%%%%%%%%%%%%%%%%%%%%%%%%%%%%%%%%%%%%%%%%%%%%%%%%%%%%%%%%%%%%%%%%%%%%%%%%%%%%%%%%
\subsection{Quantitative reconstruction}
%%%%%%%%%%%%%%%%%%%%%%%%%%%%%%%%%%%%%%%%%%%%%%%%%%%%%%%%%%%%%%%%%%%%%%%%%%%%%%%%
The embedding theorem immediately yields a quantitative reconstruction
estimate.
Fix auxiliary Riemannian metrics on
\(S\)
and
\(N,\)
and let
\(d_S, \qquad d_N\)
denote the corresponding distance functions.
\begin{theorem}
\label{thm:reconstruction-estimate}
Under the hypotheses of
Theorem~\ref{thm:response-rigidity},
there exist a neighbourhood
\(U_0\)
of \(p\) and a constant
\(C>0\)
such that
\[
d_S(p_1,p_2)
\le
C
\,d_N
\!\left(
\Theta(p_1),
\Theta(p_2)
\right)
\]
for all
\(p_1,p_2\in U_0.\)
\end{theorem}
\begin{proof}
By Theorem~\ref{thm:response-rigidity}, shrink \(U_0\) so that
\(\Theta:U_0\to V_0:=\Theta(U_0)\) is a diffeomorphism onto an embedded
submanifold. In sufficiently small normal coordinates around \(\Theta(p)\),
the ambient metric is uniformly Euclidean and \(V_0\) is a graph with bounded
derivative. Hence graph projection and parametrization are uniformly
bi-Lipschitz, so \(d_{V_0}\le C_1d_N\) on \(V_0\). The inverse
\(\Theta^{-1}:V_0\to U_0\) has bounded differential after shrinking again,
and therefore \(d_S(\Theta^{-1}y_1,\Theta^{-1}y_2)\le C_2d_{V_0}(y_1,y_2)\).
Combining the estimates proves the claim with \(C=C_1C_2\).
\end{proof}
Thus the observation map admits a locally Lipschitz inverse. Small errors in
the observation data produce proportionally small errors in the reconstructed
parameter.
%%%%%%%%%%%%%%%%%%%%%%%%%%%%%%%%%%%%%%%%%%%%%%%%%%%%%%%%%%%%%%%%%%%%%%%%%%%%%%%%
\subsection{Naturality and dependence on the observation metric}
%%%%%%%%%%%%%%%%%%%%%%%%%%%%%%%%%%%%%%%%%%%%%%%%%%%%%%%%%%%%%%%%%%%%%%%%%%%%%%%%
The response geometry depends on the response bundle morphism
\(R:TS\rightarrow\mathcal Z\)
together with the chosen fibre metric
\(h^{\mathcal Z}.\)
Replacing the observation metric changes the response metric by pullback.
\begin{proposition}
Let
\(\widetilde h^{\mathcal Z}\)
be another smooth positive-definite fibre metric on
\(\mathcal Z\).
Then the corresponding response metric is
\(\widetilde g_R = R^* \widetilde h^{\mathcal Z}.\)
If the two observation metrics are uniformly equivalent on a compact subset,
then the corresponding response metrics are uniformly equivalent there.
\end{proposition}
\begin{proof}
The first statement is immediate from the definition.
Uniform equivalence is preserved under pullback by the fixed bundle
morphism
\(R.\)
\end{proof}
Consequently, qualitative notions such as complete response, invisible
directions and observable quotient geometry depend only on the response
bundle morphism, whereas quantitative measurements depend on the chosen
observation metric.
%%%%%%%%%%%%%%%%%%%%%%%%%%%%%%%%%%%%%%%%%%%%%%%%%%%%%%%%%%%%%%%%%%%%%%%%%%%%%%%%
\section{Spectral Response Geometry}
\label{sec:spectral-response-geometry}
%%%%%%%%%%%%%%%%%%%%%%%%%%%%%%%%%%%%%%%%%%%%%%%%%%%%%%%%%%%%%%%%%%%%%%%%%%%%%%%%
The response metric introduced in the previous section is the fundamental
tensor of the response geometry. Once an auxiliary Riemannian metric has been
chosen on the parameter manifold, the response metric may equivalently be
represented by a self-adjoint bundle endomorphism. This representation is
convenient for introducing quantitative observability invariants such as
eigenvalues, singular values, determinants and condition numbers.
Throughout this section let
\((S,\{g_p\}_{p\in S})\)
be a parameterized Einstein family satisfying the standing hypotheses of
the standing hypotheses introduced above. Let
\(g_R=R^*h^{\mathcal Z}\)
denote the associated response metric.
%%%%%%%%%%%%%%%%%%%%%%%%%%%%%%%%%%%%%%%%%%%%%%%%%%%%%%%%%%%%%%%%%%%%%%%%%%%%%%%%
\subsection{The response Gram operator}
%%%%%%%%%%%%%%%%%%%%%%%%%%%%%%%%%%%%%%%%%%%%%%%%%%%%%%%%%%%%%%%%%%%%%%%%%%%%%%%%
To perform spectral analysis we fix an auxiliary Riemannian metric
\(g^S\)
on the parameter manifold \(S\).

The metric \(g^S\) is not part of the response geometry itself. Its sole
purpose is to identify tangent and cotangent bundles and thereby convert the
symmetric covariant tensor \(g_R\) into a self-adjoint endomorphism. All
spectral quantities introduced below therefore depend on the pair
\((g_R,g^S)\), whereas the response metric itself and its kernel are
independent of this auxiliary choice.
\begin{definition}[Response Gram operator]
\label{def:response-gram-operator}
The \emph{response Gram operator} is the unique smooth bundle endomorphism
\(G_R\in\Gamma(\operatorname{End}(TS))\)
satisfying
\(g^S(G_RX,Y)=g_R(X,Y)\)
for all tangent vectors \(X,Y\).
Equivalently,
\(G_R=(g^S)^{-1}g_R.\)
\end{definition}
Let
\(R^\dagger:\mathcal Z\longrightarrow TS\)
denote the fibrewise adjoint of \(R\) with respect to \(g^S\) and
\(h^{\mathcal Z}\). Since
\(g_R=R^*h^{\mathcal Z},\)
one immediately obtains the factorization
\(G_R=R^\dagger R.\)
\begin{proposition}
\label{prop:gram-properties}
The response Gram operator satisfies
\begin{enumerate}
\item
\(G_R\) is smooth and self-adjoint with respect to \(g^S\);
\item
\(G_R\) is positive semidefinite;
\item
\[
\ker G_R=\ker R=\ker g_R;
\]
\item
on the complete-response locus,
\(G_R\) is positive definite.
\end{enumerate}
\end{proposition}
\begin{proof}
The first two assertions follow immediately from the defining identity
$$g^S(G_RX,Y)=g_R(X,Y)$$
and the symmetry and positivity of the response metric.
Furthermore,
\[
g_R(X,X)
=
h^{\mathcal Z}(RX,RX)
=
\|RX\|^2,
\]
so
\[
g_R(X,X)=0
\Longleftrightarrow
RX=0.
\]
Since
$g_R(X,Y)=g^S(G_RX,Y),$
the kernels of \(G_R\), \(R\) and \(g_R\) coincide.
Finally,
\(G_R\) is positive definite precisely when its kernel is trivial, which is
equivalent to injectivity of the response bundle morphism.
\end{proof}
%%%%%%%%%%%%%%%%%%%%%%%%%%%%%%%%%%%%%%%%%%%%%%%%%%%%%%%%%%%%%%%%%%%%%%%%%%%%%%%%
\section{Constant-Rank Geometry}
\label{sec:constant-rank-geometry}
%%%%%%%%%%%%%%%%%%%%%%%%%%%%%%%%%%%%%%%%%%%%%%%%%%%%%%%%%%%%%%%%%%%%%%%%%%%%%%%%
The response metric and the associated observable quotient geometry are
defined under constant-rank assumptions on the response bundle morphism.
When, in addition, the response bundle morphism is induced by an observation
map, the constant-rank theorem yields a geometric description of the
observable and invisible directions. The purpose of this section is to
clarify these two complementary aspects.
%%%%%%%%%%%%%%%%%%%%%%%%%%%%%%%%%%%%%%%%%%%%%%%%%%%%%%%%%%%%%%%%%%%%%%%%%%%%%%%%
\subsection{Constant-rank observation maps}
%%%%%%%%%%%%%%%%%%%%%%%%%%%%%%%%%%%%%%%%%%%%%%%%%%%%%%%%%%%%%%%%%%%%%%%%%%%%%%%%
Throughout this subsection suppose that
\(U\subset S\)
is open and that there exist
\begin{itemize}
\item
a Riemannian manifold \((N,h^N)\);
\item
a smooth observation map
\(\Theta:U\longrightarrow N;\)
\item
a vector-bundle isomorphism
\(\Psi: \mathcal Z|_U \longrightarrow \Theta^*TN\)
such that
\(\Psi\circ R=d\Theta.\)
\end{itemize}
Assume further that \(d\Theta\) has constant rank \(r\) on \(U\).
The classical constant-rank theorem (see, for example, \cite{LeeSmooth}) implies that every point
\(p\in U\) has a neighbourhood
\(U_p\subset U\)
on which the fibres and local image of \(\Theta\) have the expected smooth
structure.
\begin{proposition}
\label{prop:constant-rank-observation}
After shrinking \(U_p\) if necessary, the following statements hold.
\begin{enumerate}
\item
For every \(q\in U_p\), the level set
\((\Theta|_{U_p})^{-1}(\Theta(q))\)
is an embedded submanifold of \(U_p\) of dimension
\(\dim S-r.\)
\item
Its tangent space at \(q\) is
\[
T_q\!\left(
(\Theta|_{U_p})^{-1}(\Theta(q))
\right)
=
\ker d\Theta_q
=
\ker R_q.
\]
\item
The local image
\(\Theta(U_p)\)
is an embedded \(r\)-dimensional submanifold of \(N\).
\item
The differential induces a canonical vector-bundle isomorphism
\[
\overline{d\Theta}:
TU_p/\ker R
\longrightarrow
T\Theta(U_p),
\qquad
[X]\longmapsto d\Theta(X).
\]
\end{enumerate}
If, in addition,
\(h^{\mathcal Z} = \Psi^*(\Theta^*h^N),\)
then \(\overline{d\Theta}\) is an isometry from the observable quotient
metric \(\overline g_R\) to the Riemannian metric on
\(T\Theta(U_p)\) induced by \(h^N\).
\end{proposition}
\begin{proof}
Since
\(d\Theta=\Psi\circ R\)
and \(\Psi\) is fibrewise invertible,
\(\ker d\Theta_q=\ker R_q\)
for every \(q\in U\).
The first three assertions are the local constant-rank theorem applied to
\(\Theta:U\longrightarrow N.\)
Because \(d\Theta\) vanishes precisely on \(\ker R\), it descends fibrewise
to a linear map
$$\overline{d\Theta}_q: T_qS/\ker R_q \longrightarrow T_{\Theta(q)}\Theta(U_p).$$
The constant-rank theorem shows that this map is an isomorphism at every
\(q\in U_p\), and the resulting family is smooth. Hence
\(\overline{d\Theta}: TU_p/\ker R \longrightarrow T\Theta(U_p)\)
is a vector-bundle isomorphism.
Finally, assume that \(\Psi\) is fibrewise isometric. For tangent classes
\([X],[Y]\in T_qS/\ker R_q\),
\[
\begin{aligned}
\overline g_R([X],[Y])
&=
h_q^{\mathcal Z}(R_qX,R_qY)\\
&=
h^N_{\Theta(q)}
\bigl(
\Psi_qR_qX,\Psi_qR_qY
\bigr)\\
&=
h^N_{\Theta(q)}
\bigl(
d\Theta_qX,d\Theta_qY
\bigr).
\end{aligned}
\]
Thus \(\overline{d\Theta}\) is an isometry onto the tangent bundle of the
local image equipped with the metric induced from \(h^N\).
\end{proof}
\begin{remark}
The existence of an observation map \(\Theta\) realizing the response bundle
morphism is an additional hypothesis. For a general bundle morphism
\(R:TS\longrightarrow\mathcal Z,\)
the invisible bundle need not be integrable.
\end{remark}
%%%%%%%%%%%%%%%%%%%%%%%%%%%%%%%%%%%%%%%%%%%%%%%%%%%%%%%%%%%%%%%%%%%%%%%%%%%%%%%%
\subsection{Rank degeneration}
%%%%%%%%%%%%%%%%%%%%%%%%%%%%%%%%%%%%%%%%%%%%%%%%%%%%%%%%%%%%%%%%%%%%%%%%%%%%%%%%
The geometric constructions developed in the preceding sections require the
response bundle morphism to have locally constant rank. It is therefore
natural to distinguish failure of constant rank from failure of complete
response.
\begin{definition}
The \emph{rank-singular set} is
\[
\Sigma_{\mathrm{rank}}
=
\{
p\in S:
\operatorname{rank}R
\text{ is not locally constant at }
p
\}.
\]
\end{definition}
On the complement
\(S\setminus\Sigma_{\mathrm{rank}},\)
the invisible bundle, observable bundle and observable quotient bundle are
well-defined smooth vector bundles.
By contrast, the failure of complete response is measured by the
response-degeneracy locus
$$S_{\mathrm{deg}} = \{ p\in S: R_p \text{ is not injective} \}.$$
The two notions are independent.
Indeed, a response bundle morphism may have constant rank without being
injective, in which case the observable quotient geometry is perfectly
well-defined although invisible directions persist.

\subsection{Spectral characterization}
%%%%%%%%%%%%%%%%%%%%%%%%%%%%%%%%%%%%%%%%%%%%%%%%%%%%%%%%%%%%%%%%%%%%%%%%%%%%%%%%
Fix the auxiliary Riemannian metric introduced in
Section~\ref{sec:spectral-response-geometry}, and let $G_R$ denote the associated response Gram operator. Assume $\dim S>0$ so that the smallest singular value of $R_p$ is defined for every $p\in S$.

For every parameter $p\in S$, the following conditions are equivalent:
\begin{enumerate}
\item $p\in S_{\mathrm{deg}}$;
\item $\ker R_p\neq{0}$;
\item $\ker (G_R)_p\neq{0}$;
\item $(g_R)_p$ is degenerate;
\item $\sigma_{\min}(R_p)=0$.
\end{enumerate}

Thus the rank-singular set records changes in the observable dimension,
whereas the response-degeneracy locus records failure of complete response.
The two sets coincide only under additional hypotheses.

If the response bundle morphism is real analytic, every rank-defect locus is
a closed real-analytic subset of the parameter manifold. No smooth manifold
structure is asserted without further transversality assumptions.

\section{Covariant Variation of Response Geometry}
\label{sec:response-transport}
The response metric varies over the parameter manifold together with the
response bundle morphism. To differentiate these objects invariantly, one
must choose connections on the tangent and observation bundles. The
constructions of this section are therefore relative to these auxiliary
choices.

Fix an auxiliary Riemannian metric \(g^S\) on \(S\), let
\(\nabla^S\) denote its Levi--Civita connection, and let
\(\nabla^{\mathcal Z}\) be a metric connection on
\((\mathcal Z,h^{\mathcal Z})\). These connections induce the standard
connection on
\(\operatorname{Hom}(TS,\mathcal Z) \cong T^*S\otimes\mathcal Z;\)
see, for example, \cite{KobayashiNomizu,Petersen}.
%%%%%%%%%%%%%%%%%%%%%%%%%%%%%%%%%%%%%%%%%%%%%%%%%%%%%%%%%%%%%%%%%%%%%%%%%%%%%%%%
\subsection{Covariant derivative of the response bundle morphism}
\label{subsec:transport-response-section}
%%%%%%%%%%%%%%%%%%%%%%%%%%%%%%%%%%%%%%%%%%%%%%%%%%%%%%%%%%%%%%%%%%%%%%%%%%%%%%%%
\begin{definition}[Response derivative tensor]
The \emph{response derivative tensor} is
\[
\mathcal T:=\nabla R
\in
\Gamma\!\left(
T^*S\otimes\operatorname{Hom}(TS,\mathcal Z)
\right).
\]
Equivalently, for vector fields \(X,Y\) on \(S\),
\[
\mathcal T(X)Y
=
(\nabla_XR)(Y)
=
\nabla^{\mathcal Z}_X(RY)
-
R(\nabla^S_XY).
\]
\end{definition}
Thus \(\mathcal T\) measures the covariant variation of the response bundle
morphism relative to the chosen connections. It is not determined by the
first-order response theory alone.

The first-order detection theory of \cite{EDP} supplies the pointwise response
data under its stated hypotheses, but no formula for \(\nabla R\) is imported
here from that work. In a concrete Einstein--probe system, expressing
\(\mathcal T\) through differentiated Jacobi, forcing, event-location or
observation operators is a separate second-variation problem.
%%%%%%%%%%%%%%%%%%%%%%%%%%%%%%%%%%%%%%%%%%%%%%%%%%%%%%%%%%%%%%%%%%%%%%%%%%%%%%%%
\subsection{Covariant variation of the response metric}
\label{subsec:transport-response-metric}
%%%%%%%%%%%%%%%%%%%%%%%%%%%%%%%%%%%%%%%%%%%%%%%%%%%%%%%%%%%%%%%%%%%%%%%%%%%%%%%%
The first variation of the response metric is determined algebraically by
\(\mathcal T\).
\begin{proposition}[Covariant derivative of the response metric]
\label{thm:transport-response-metric}
For all vector fields \(X,Y,Z\) on \(S\),
\[
(\nabla_X g_R)(Y,Z)
=
h^{\mathcal Z}\!\left(\mathcal T(X)Y,RZ\right)
+
h^{\mathcal Z}\!\left(RY,\mathcal T(X)Z\right).
\]
\end{proposition}
\begin{proof}
Since
\(g_R(Y,Z) = h^{\mathcal Z}(RY,RZ),\)
metric compatibility of \(\nabla^{\mathcal Z}\) gives
\[
\begin{aligned}
X\!\left[g_R(Y,Z)\right]
&=
h^{\mathcal Z}\!\left(
\nabla^{\mathcal Z}_X(RY),RZ
\right)
+
h^{\mathcal Z}\!\left(
RY,\nabla^{\mathcal Z}_X(RZ)
\right).
\end{aligned}
\]
Using
$$\nabla^{\mathcal Z}_X(RY) = \mathcal T(X)Y+R(\nabla^S_XY)$$
and the analogous identity for \(Z\), and subtracting
$$g_R(\nabla^S_XY,Z)+g_R(Y,\nabla^S_XZ),$$
yields the stated formula.
\end{proof}
In particular, the variation of \(g_R\) is completely determined by the
response bundle morphism and its covariant derivative once the auxiliary
connections have been fixed.
%%%%%%%%%%%%%%%%%%%%%%%%%%%%%%%%%%%%%%%%%%%%%%%%%%%%%%%%%%%%%%%%%%%%%%%%%%%%%%%%
\subsection{Covariant variation of the response Gram operator}
\label{subsec:transport-gram-operator}
%%%%%%%%%%%%%%%%%%%%%%%%%%%%%%%%%%%%%%%%%%%%%%%%%%%%%%%%%%%%%%%%%%%%%%%%%%%%%%%%
Let \(G_R\) denote the response Gram operator defined by
\(g^S(G_RX,Y)=g_R(X,Y).\)
Since \(\nabla^Sg^S=0\),
\(\nabla G_R = (g^S)^{-1}\nabla g_R.\)
If \(p(t)\) is a smooth path in \(S\), parallel transport with respect to
\(\nabla^S\) identifies the tangent spaces along the path. We write
\(\frac{D}{dt}G_R\)
for the corresponding covariant derivative.
\begin{proposition}[Variation of relative spectral invariants]
\label{thm:transport-spectral-observability}
Suppose that \(\mu_i(t)>0\) is a simple eigenvalue of \(G_R(t)\), with
\(g^S\)-unit eigenvector \(u_i(t)\), and set
\(\sigma_i(t)=\sqrt{\mu_i(t)}.\)
Then
\[
\dot\mu_i
=
g^S\!\left(
\frac{D}{dt}G_R\,u_i,u_i
\right)
\]
and
\[
\dot\sigma_i
=
\frac{1}{2\sigma_i}
g^S\!\left(
\frac{D}{dt}G_R\,u_i,u_i
\right).
\]
On the positive-dimensional part of the complete-response locus,
\[
\frac{d}{dt}\log\det G_R
=
\operatorname{tr}\!\left(
G_R^{-1}\frac{D}{dt}G_R
\right).
\]
If the largest and smallest singular values are simple, then
\[
\frac{d}{dt}\log\kappa_R
=
\frac{\dot\sigma_{\max}}{\sigma_{\max}}
-
\frac{\dot\sigma_{\min}}{\sigma_{\min}}.
\]
\end{proposition}
\begin{proof}
After parallel transport with respect to \(\nabla^S\), all tangent spaces are identified with a fixed Euclidean vector space, so that \(G_R(t)\) becomes a smooth family of self-adjoint matrices. The first two identities are the classical Rayleigh quotient formulas for simple eigenvalues. On the complete-response locus, \(G_R(t)\) is invertible, so Jacobi's determinant formula gives
\[
\frac{d}{dt}\log\det G_R
=
\operatorname{tr}\!\left(
G_R^{-1}\frac{D}{dt}G_R
\right).
\]
Finally, the condition-number identity is obtained by differentiating
\(\log\kappa_R = \log\sigma_{\max} - \log\sigma_{\min},\)
which is valid while the extremal singular values remain simple.
See \cite{HornJohnson,Kato}.
\end{proof}
No differentiability of individual eigenvalue branches is asserted at
eigenvalue crossings. Likewise, formulas involving
\(\sigma_{\min}^{-1}\) or \(G_R^{-1}\) are restricted to regions where the
corresponding quantities are positive.

%%%%%%%%%%%%%%%%%%%%%%%%%%%%%%%%%%%%%%%%%%%%%%%%%%%%%%%%%%%%%%%%%%%%%%%%%%%%%%%%
\subsection{Variation of the response volume density}
\label{subsec:transport-response-volume}
%%%%%%%%%%%%%%%%%%%%%%%%%%%%%%%%%%%%%%%%%%%%%%%%%%%%%%%%%%%%%%%%%%%%%%%%%%%%%%%%
On the complete-response locus, define the response volume density relative
to the auxiliary metric \(g^S\) by
\(d\operatorname{vol}_{g_R} = \sqrt{\det G_R}\, d\operatorname{vol}_{g^S}.\)
\begin{corollary}
Along a smooth path contained in the positive-dimensional
complete-response locus,
\[
\frac{D}{dt}d\operatorname{vol}_{g_R}
=
\frac12
\operatorname{tr}\!\left(
G_R^{-1}\frac{D}{dt}G_R
\right)
d\operatorname{vol}_{g_R},
\]
where the derivative on densities is taken with respect to the connection
induced by \(\nabla^S\).
\end{corollary}
\begin{proof}
The volume density of \(g^S\) is parallel under its Levi--Civita connection.
The result therefore follows by differentiating
\(\sqrt{\det G_R}\)
and applying Proposition~\ref{thm:transport-spectral-observability}.
\end{proof}
%%%%%%%%%%%%%%%%%%%%%%%%%%%%%%%%%%%%%%%%%%%%%%%%%%%%%%%%%%%%%%%%%%%%%%%%%%%%%%%%
\subsection{A conditional closure hypothesis}
\label{subsec:closed-transport}
%%%%%%%%%%%%%%%%%%%%%%%%%%%%%%%%%%%%%%%%%%%%%%%%%%%%%%%%%%%%%%%%%%%%%%%%%%%%%%%%
The general response theory does not provide an evolution or closure
equation for \(R\). In a particular model one may, however, impose or derive
additional structure.
Suppose that there exists a smooth bundle map
\[
\Phi:
\operatorname{Hom}(TS,\mathcal Z)
\longrightarrow
T^*S\otimes\operatorname{Hom}(TS,\mathcal Z)
\]
such that
\(\nabla R=\Phi(R).\)
Then Proposition~\ref{thm:transport-response-metric} becomes
\[
(\nabla_Xg_R)(Y,Z)
=
h^{\mathcal Z}\!\left(
\Phi(R)(X)Y,RZ
\right)
+
h^{\mathcal Z}\!\left(
RY,\Phi(R)(X)Z
\right).
\]
This is a conditional closure hypothesis, not a consequence of the abstract
response geometry or of the first-order detection theory in \cite{EDP}. Establishing such a
formula for a concrete Einstein--probe system requires an independent
second-order analysis of the corresponding Einstein, Jacobi and observation
equations.
\section{Minimal Scalar Probe Families}
\label{sec:minimal-probes}
%%%%%%%%%%%%%%%%%%%%%%%%%%%%%%%%%%%%%%%%%%%%%%%%%%%%%%%%%%%%%%%%%%%%%%%%%%%%%%%%
The response metric describes infinitesimal observability once a probe
package has been fixed. A complementary question is how many scalar
measurements obtained by admissible linear postprocessing of a fixed master
observation map are required to retain infinitesimal completeness.
At a point of an \(m\)-dimensional parameter manifold, every scalar
observation map with injective differential must have at least \(m\)
components. Under an explicit richness hypothesis on the admissible scalar
postprocessings, this lower bound is attained. A minimal complete family then
provides local reconstruction and quantitative stability.
%%%%%%%%%%%%%%%%%%%%%%%%%%%%%%%%%%%%%%%%%%%%%%%%%%%%%%%%%%%%%%%%%%%%%%%%%%%%%%%%
\subsection{Master observations and admissible scalar postprocessings}
%%%%%%%%%%%%%%%%%%%%%%%%%%%%%%%%%%%%%%%%%%%%%%%%%%%%%%%%%%%%%%%%%%%%%%%%%%%%%%%%
Let
\((S,\{g_p\}_{p\in S})\)
be a parameterized Einstein family in the sense of
the standing hypotheses introduced above. Fix
\(p\in S\)
and set
\[
V:=T_pS,
\qquad
m:=\dim V.
\]
Suppose that, on an open neighbourhood \(U\subset S\) of \(p\), the chosen
probe package determines a \(C^1\) master observation map
\(\mathcal A:U\longrightarrow Y\)
into a Banach space \(Y\). Write
\[
L:=d\mathcal A_p:V\longrightarrow Y.
\]
Let \(E\) be a Banach space together with a continuous injective linear map
\(\iota:E\longrightarrow Y^*.\)
We identify \(E\) with its image under \(\iota\) and call its elements
\emph{admissible scalar postprocessings}.
For
\(\boldsymbol\ell = (\ell_1,\ldots,\ell_N) \in E^N,\)
define
\[
\Theta_{\boldsymbol\ell}
:=
\bigl(
\ell_1\circ\mathcal A,\ldots,\ell_N\circ\mathcal A
\bigr)
:
U\longrightarrow\mathbb R^N.
\]
Since every \(\ell_i\) is continuous and linear,
\(\Theta_{\boldsymbol\ell}\) is \(C^1\), with
\[
d\Theta_{\boldsymbol\ell,p}(v)
=
\bigl(
\ell_1(Lv),\ldots,\ell_N(Lv)
\bigr),
\qquad
v\in V.
\]
\begin{definition}[Admissible scalar probe number]
\label{def:admissible-scalar-probe-number}
The \emph{admissible scalar probe number} at \(p\) is
\[
\nu_E(p)
:=
\min
\left\{
N\in\mathbb N_0:
\begin{array}{c}
\text{there exists }
\boldsymbol\ell\in E^N
\text{ for which}\\[1mm]
d\Theta_{\boldsymbol\ell,p}
\text{ is injective}
\end{array}
\right\},
\]
provided that this set is nonempty. If no finite infinitesimally complete
admissible family exists, we set
\[
\nu_E(p):=+\infty.
\]
For \(N=0\), we use the conventions
\[
E^0=\{()\},
\qquad
\mathbb R^0=\{0\}.
\]
\end{definition}
Equip the finite-dimensional subspace
\(L(V)\subset Y\)
with the norm induced from \(Y\), and equip \(L(V)^*\) with the corresponding
dual norm. The relevant richness hypothesis is surjectivity of the continuous
linear restriction map
\[
\rho_p:E\longrightarrow L(V)^*,
\qquad
\rho_p(\ell):=\ell|_{L(V)}.
\]
Thus every continuous linear functional on the finite-dimensional master
response space \(L(V)\) is realized by an admissible scalar postprocessing.
%%%%%%%%%%%%%%%%%%%%%%%%%%%%%%%%%%%%%%%%%%%%%%%%%%%%%%%%%%%%%%%%%%%%%%%%%%%%%%%%
\begin{remark}
The invariant $\nu_E(p)$ is pointwise. Even when its value equals $\dim T_pS$ at every point, the particular minimizing family of scalar postprocessings need not vary continuously with $p$ and need not define a locally constant choice.
\end{remark}
\subsection{The minimal scalar probe theorem}
%%%%%%%%%%%%%%%%%%%%%%%%%%%%%%%%%%%%%%%%%%%%%%%%%%%%%%%%%%%%%%%%%%%%%%%%%%%%%%%%
\begin{theorem}[Minimal scalar probe theorem]
\label{thm:minimal-scalar-probe}
Assume that
\(\rho_p:E\longrightarrow L(V)^*\)
is surjective. Then the following statements hold.
\begin{enumerate}
\item
A finite infinitesimally complete admissible scalar family exists if and only
if
\(L=d\mathcal A_p\)
is injective.
\item
If \(L\) is injective, then
\[
\nu_E(p)
=
m
=
\dim T_pS.
\]
In particular, if \(m=0\), then the empty scalar family is infinitesimally
complete and
\(\nu_E(p)=0.\)
\item
Suppose that \(L\) is injective and \(N\ge m\). Then
\[
\mathcal C_{p,N}
:=
\left\{
\boldsymbol\ell\in E^N:
d\Theta_{\boldsymbol\ell,p}
\text{ is injective}
\right\}
\]
is open and dense in \(E^N\) with respect to its Banach-space norm topology.
\item
Suppose that \(m\ge1\), \(N=m\), and
\(\boldsymbol\ell\in\mathcal C_{p,m}.\)
Then \(\Theta_{\boldsymbol\ell}\) is a local \(C^1\)-diffeomorphism at
\(p\). Consequently, there exists an open neighbourhood
\(U_0\subset U\) of \(p\) such that
\[
\Theta_{\boldsymbol\ell}^{-1}
\bigl(
\Theta_{\boldsymbol\ell}(p)
\bigr)
\cap U_0
=
\{p\}.
\]
\item
Equip \(S\) with an auxiliary Riemannian metric \(g^S\), and let \(d_S\)
denote the associated Riemannian distance. Under the hypotheses of
part~\textup{(4)}, the neighbourhood \(U_0\) after shrinking the neighbourhood if necessary, may be chosen so that there
exist constants
\(0<c\le C<\infty\)
such that
\[
c\,d_S(p_1,p_2)
\le
\left|
\Theta_{\boldsymbol\ell}(p_1)
-
\Theta_{\boldsymbol\ell}(p_2)
\right|
\le
C\,d_S(p_1,p_2)
\]
for all
\(p_1,p_2\in U_0.\)
Thus every minimal infinitesimally complete admissible family induces a
local bi-Lipschitz chart on the parameter manifold near \(p\).
\end{enumerate}
\end{theorem}
The injectivity of the master response and the richness of the admissible
postprocessing class are logically independent hypotheses.
\begin{proof}
Suppose first that \(L\) is not injective. Choose
\(0\neq v\in\ker L.\)
For every \(N\in\mathbb N_0\) and every
\(\boldsymbol\ell\in E^N,\)
one has
\[
d\Theta_{\boldsymbol\ell,p}(v)
=
\bigl(
\ell_1(Lv),\ldots,\ell_N(Lv)
\bigr)
=
0.
\]
Hence no admissible scalar family can be infinitesimally complete.
Assume conversely that \(L\) is injective. Then
\(L:V\longrightarrow L(V)\)
is a linear isomorphism and
\(\dim L(V)=m.\)
Choose a basis
\(\lambda_1,\ldots,\lambda_m\)
of \(L(V)^*\). By surjectivity of \(\rho_p\), choose
\(\ell_1,\ldots,\ell_m\in E\)
such that
\(\ell_i|_{L(V)}=\lambda_i.\)
The resulting linear map
$$v \longmapsto \bigl( \ell_1(Lv),\ldots,\ell_m(Lv) \bigr)$$
is an isomorphism
\(V\longrightarrow\mathbb R^m.\)
Hence \(m\) admissible scalar observations suffice.
If \(m=0\), the unique map
\(\{0\}\longrightarrow\mathbb R^0\)
is injective, so the empty family is infinitesimally complete.
Conversely, if
\(d\Theta_{\boldsymbol\ell,p} : V\longrightarrow\mathbb R^N\)
is injective, finite-dimensional linear algebra gives
\(N\ge\dim V=m.\)
Therefore
\(\nu_E(p)=m.\)
Now let \(N\ge m\), and define
$\mathcal P_N:E^N \longrightarrow \operatorname{Hom}(V,\mathbb R^N)$
by
$$\mathcal P_N(\boldsymbol\ell)(v) = \bigl( \ell_1(Lv),\ldots,\ell_N(Lv) \bigr).$$
Equip
\(\operatorname{Hom}(V,\mathbb R^N)\)
with any operator norm induced by norms on \(V\) and \(\mathbb R^N\).
Since this space is finite dimensional, all such norms are equivalent and
induce the same topology.
The map \(\mathcal P_N\) is continuous and linear. To prove surjectivity, let
\(T=(T_1,\ldots,T_N) \in \operatorname{Hom}(V,\mathbb R^N),\)
where \(T_i\in V^*\). Since
\(L:V\longrightarrow L(V)\)
is an isomorphism, each
\[
\lambda_i:=T_i\circ L^{-1}
\]
belongs to \(L(V)^*\). Surjectivity of \(\rho_p\) yields
\(\ell_i\in E\) satisfying
\(\ell_i|_{L(V)}=\lambda_i.\)
Hence
\(\mathcal P_N(\ell_1,\ldots,\ell_N)=T,\)
so \(\mathcal P_N\) is surjective.
Both \(E^N\) and
\(\operatorname{Hom}(V,\mathbb R^N)\) are Banach spaces. Therefore the Open
Mapping Theorem implies that \(\mathcal P_N\) is an open map.
If \(m=0\), every linear map
\(V=\{0\}\longrightarrow\mathbb R^N\)
is injective, and hence
\(\mathcal C_{p,N}=E^N.\)
Assume henceforth in the genericity argument that \(m\ge1\).
The set
$$\operatorname{Inj}(V,\mathbb R^N) \subset \operatorname{Hom}(V,\mathbb R^N)$$
is open and dense for \(N\ge m\). Choose arbitrary bases of \(V\) and
\(\mathbb R^N\). Under these identifications,
\(\operatorname{Inj}(V,\mathbb R^N)\) is the set of
\(N\times m\) matrices of column rank \(m\). Openness follows from the
nonvanishing of an \(m\times m\) minor, while density follows from
arbitrarily small perturbations to full column rank.
Therefore
\[
\mathcal C_{p,N}
=
\mathcal P_N^{-1}
\bigl(
\operatorname{Inj}(V,\mathbb R^N)
\bigr)
\]
is open.
To prove density, let
\(\mathcal O\subset E^N\)
be nonempty and open. Since \(\mathcal P_N\) is open and surjective,
\(\mathcal P_N(\mathcal O)\)
is a nonempty open subset of
\(\operatorname{Hom}(V,\mathbb R^N)\). It therefore intersects the dense set
\(\operatorname{Inj}(V,\mathbb R^N)\). Consequently,
\(\mathcal O\cap\mathcal C_{p,N}\neq\varnothing.\)
Thus \(\mathcal C_{p,N}\) is dense.
Suppose now that \(m\ge1\), \(N=m\), and
\(\boldsymbol\ell\in\mathcal C_{p,m}.\)
Then
\(d\Theta_{\boldsymbol\ell,p} : T_pS\longrightarrow\mathbb R^m\)
is a linear isomorphism. The inverse function theorem therefore gives open
neighbourhoods
\[
p\in U_1\subset U,
\qquad
\Theta_{\boldsymbol\ell}(p)\in V_1\subset\mathbb R^m,
\]
such that
\(\Theta_{\boldsymbol\ell}|_{U_1} : U_1\longrightarrow V_1\)
is a \(C^1\)-diffeomorphism. This proves the local uniqueness statement.
For the quantitative estimate, choose a relatively compact strongly convex
normal neighbourhood
\(p\in W\Subset U_1.\)
By continuity of the inverse map, there exists an open Euclidean ball
\(\Theta_{\boldsymbol\ell}(p)\in V_0\Subset V_1\)
such that
\[
U_0
:=
\bigl(
\Theta_{\boldsymbol\ell}|_{U_1}
\bigr)^{-1}(V_0)
\subset W.
\]
Because \(\overline W\subset U_1\) is compact, there exists \(C>0\) such that
\[
\|d\Theta_{\boldsymbol\ell,x}\|_{g^S,\mathrm{eucl}}
\le C
\]
for all \(x\in\overline W\). For \(p_1,p_2\in U_0\), the minimizing
geodesic contained in the strongly convex neighbourhood \(W\) has length
\(d_S(p_1,p_2)\). Integrating \(d\Theta_{\boldsymbol\ell}\) along this
geodesic gives
\[
\left|
\Theta_{\boldsymbol\ell}(p_1)
-
\Theta_{\boldsymbol\ell}(p_2)
\right|
\le
C\,d_S(p_1,p_2).
\]
The inverse map
\(\bigl( \Theta_{\boldsymbol\ell}|_{U_1} \bigr)^{-1} : V_1\longrightarrow U_1\)
has bounded differential on \(\overline{V_0}\). Thus there exists \(C'>0\)
such that
\[
\left\|
d\bigl(
\Theta_{\boldsymbol\ell}|_{U_1}
\bigr)^{-1}_y
\right\|_{\mathrm{eucl},g^S}
\le C'
\]
for all \(y\in\overline{V_0}\).
Since \(V_0\) is convex, the Euclidean line segment joining
\(\Theta_{\boldsymbol\ell}(p_1)\) and
\(\Theta_{\boldsymbol\ell}(p_2)\) lies in \(V_0\). Its image under the inverse
map is a piecewise \(C^1\) curve \(\gamma\) in \(U_0\) joining \(p_1\) to
\(p_2\). The derivative bound gives
\[
\operatorname{Length}_{g^S}(\gamma)
\le
C'
\left|
\Theta_{\boldsymbol\ell}(p_1)
-
\Theta_{\boldsymbol\ell}(p_2)
\right|.
\]
Since Riemannian distance is bounded above by the length of every joining
curve,
\[
d_S(p_1,p_2)
\le
\operatorname{Length}_{g^S}(\gamma),
\]
and hence
\[
d_S(p_1,p_2)
\le
C'
\left|
\Theta_{\boldsymbol\ell}(p_1)
-
\Theta_{\boldsymbol\ell}(p_2)
\right|.
\]
Set $c:=(C')^{-1}$, which
completes the proof.
\end{proof}

\begin{corollary}[Full-dual postprocessings]
\label{cor:full-dual-minimal-probes}
Suppose that
\(E=Y^*.\)
Then the restriction map
\(Y^*\longrightarrow L(V)^*\)
is surjective. Consequently, if \(d\mathcal A_p\) is injective, then
\(\nu_{Y^*}(p)=\dim T_pS.\)
\end{corollary}
\begin{proof}
The finite-dimensional subspace \(L(V)\subset Y\) is closed. Every
continuous linear functional on \(L(V)\) extends to a continuous linear
functional on \(Y\) by the Hahn--Banach theorem.
\end{proof}
\begin{corollary}[Minimal complete scalar response metric]
\label{cor:minimal-complete-response-metric}
Under the hypotheses of
Theorem~\ref{thm:minimal-scalar-probe}, suppose that \(L\) is injective and
choose
\(\boldsymbol\ell\in\mathcal C_{p,m}.\)
Define
\[
g_{\boldsymbol\ell}
:=
\Theta_{\boldsymbol\ell}^*
\langle\cdot,\cdot\rangle_{\mathbb R^m}.
\]
Then \(g_{\boldsymbol\ell}\) is positive definite at \(p\). After shrinking
the neighbourhood of \(p\), it is a Riemannian metric.
Moreover,
\(m=\dim T_pS\)
is the smallest number of admissible linear scalar postprocessings of the
fixed master observation map whose pullback Euclidean metric can be positive
definite at \(p\).
\end{corollary}
\begin{proof}
For \(X\in T_pS\),
\[
(g_{\boldsymbol\ell})_p(X,X)
=
\left|
d\Theta_{\boldsymbol\ell,p}(X)
\right|^2.
\]
Thus positive definiteness at \(p\) is equivalent to injectivity of
\(d\Theta_{\boldsymbol\ell,p}\). Injectivity of the differential is an open
condition, so the same fixed family
\(\boldsymbol\ell\) defines a positive-definite pullback metric after
shrinking the neighbourhood.
If a finite family of \(N\) admissible linear scalar postprocessings produces
a map into \(\mathbb R^N\) with injective differential at \(p\), then
\(N\ge\dim T_pS=m.\)
Hence \(m\) is minimal.
\end{proof}

\begin{remark}[Scope of the abstract theorem]
\label{rem:minimal-probe-detection}
Theorem~\ref{thm:minimal-scalar-probe} is an abstract finite-dimensional
statement based on a functional-analytic surjectivity hypothesis. Its
application to a specific Einstein--probe system requires independent
verification of
\begin{enumerate}
\item
a \(C^1\) master observation map
\(\mathcal A:U\longrightarrow Y;\)
\item
injectivity of
\(d\mathcal A_p:T_pS\longrightarrow Y;\)
\item
surjectivity of the admissible restriction map
\(E\longrightarrow d\mathcal A_p(T_pS)^*.\)
\end{enumerate}
No one of these properties is attributed to \cite{EDP} unless it has been
established there for the particular probe package and admissible scalar
observation class under consideration.
\end{remark}
\begin{remark}[Pointwise character]
\label{rem:minimal-probe-pointwise}
Theorem~\ref{thm:minimal-scalar-probe} is pointwise. It does not assert that
one fixed \(m\)-tuple selected at \(p\) is minimal throughout a neighbourhood
or on an entire parameterized Einstein family.

The same fixed tuple remains infinitesimally complete on some sufficiently
small neighbourhood of \(p\), because injectivity of its differential is an
open condition. Uniform or global minimality requires additional hypotheses
and may fail even when pointwise minimality holds everywhere.
\end{remark}
%%%%%%%%%%%%%%%%%%%%%%%%%%%%%%%%%%%%%%%%%%%%%%%%%%%%%%%%%%%%%%%%%%%%%%%%%%%%%%%%
\section{Finite-Dimensional Illustrations and Conditional Cohomogeneity-One Applications}
\label{sec:applications}
%%%%%%%%%%%%%%%%%%%%%%%%%%%%%%%%%%%%%%%%%%%%%%%%%%%%%%%%%%%%%%%%%%%%%%%%%%%%%%%%
The abstract response geometry developed in the preceding sections applies to
any parameterized Einstein family satisfying the standing hypotheses. The
computations below serve a different purpose. They illustrate the resulting
finite-dimensional formulas on shooting-coordinate spaces and record two
conditional extensions to cohomogeneity-one settings.

The distinction is important: a shooting parameter manifold of one-sided
Einstein germs or propagated local solutions is not, in general, a
parameterized family of globally closing Einstein metrics. Accordingly, the
tensor constructed on such a shooting space is a parameter-response tensor.
It agrees with the response metric of the preceding theory only after a
separate argument identifies the parameter space, or a smooth subfamily of
it, with a parameterized family of globally defined normalized Einstein
metrics.
%%%%%%%%%%%%%%%%%%%%%%%%%%%%%%%%%%%%%%%%%%%%%%%%%%%%%%%%%%%%%%%%%%%%%%%%%%%%%%%%
\subsection{A two-parameter response tensor}
\label{subsec:two-parameter-response}
%%%%%%%%%%%%%%%%%%%%%%%%%%%%%%%%%%%%%%%%%%%%%%%%%%%%%%%%%%%%%%%%%%%%%%%%%%%%%%%%
Let \(P\) be a smooth two-dimensional parameter manifold of normalized
one-sided Einstein germs, or of Einstein solutions propagated to a prescribed
matching hypersurface. Choose local coordinates
\(x=(x^1,x^2)\)
on an open set \(U\subset P\), centred at \(p_0\), and equip \(U\) with the
auxiliary Euclidean metric
\(g^P=(dx^1)^2+(dx^2)^2.\)
Let
\(\Theta=(O_1,\ldots,O_N):U\longrightarrow\mathbb R^N\)
be a smooth family of scalar observations, and equip \(\mathbb R^N\) with
the weighted Euclidean metric
\[
h^{\mathcal Z}
=
\sum_{\alpha=1}^N w_\alpha\,dz_\alpha^2,
\qquad
w_\alpha>0.
\]
Define the parameter-response tensor
\[
g_R:=\Theta^*h^{\mathcal Z}.
\]
If \(U\) itself parameterizes a smooth family of globally defined normalized
Einstein metrics and the observation map is the one used in the abstract
response construction, then \(g_R\) is the response metric introduced
earlier. In the general shooting situation it measures only the
distinguishability of shooting variations.
\begin{theorem}[Explicit two-parameter response tensor]
\label{thm:explicit-two-parameter-response-metric}
Set
\[
J(p)=D\Theta(p)
=
\begin{pmatrix}
\partial_1O_1 & \partial_2O_1\\
\vdots & \vdots\\
\partial_1O_N & \partial_2O_N
\end{pmatrix},
\qquad
W=\operatorname{diag}(w_1,\ldots,w_N).
\]
In the coordinate frame \((\partial_1,\partial_2)\),
\([g_R]_x=J^{\mathsf T}WJ.\)
Writing
\[
A=\sum_{\alpha=1}^N w_\alpha(\partial_1O_\alpha)^2,
\qquad
B=\sum_{\alpha=1}^N w_\alpha(\partial_1O_\alpha)(\partial_2O_\alpha),
\]
and
\[
C=\sum_{\alpha=1}^N w_\alpha(\partial_2O_\alpha)^2,
\]
one has
\[
[g_R]_x=
\begin{pmatrix}
A&B\\
B&C
\end{pmatrix}.
\]
Moreover,
\[
\det[g_R]_x
=
\sum_{1\le \alpha<\beta\le N}
w_\alpha w_\beta
\left(
\partial_1O_\alpha\,\partial_2O_\beta
-
\partial_2O_\alpha\,\partial_1O_\beta
\right)^2.
\]
Consequently, at every \(p\in U\), the following are equivalent:
\begin{enumerate}
\item \(g_R\) is positive definite at \(p\);
\item \(D\Theta(p)\) has rank two;
\item at least one \(2\times2\) minor of \(D\Theta(p)\) is nonzero;
\item \(\det[g_R]_x(p)>0\).
\end{enumerate}
Because the chosen coordinates make
\(g^P=(dx^1)^2+(dx^2)^2\), the Gram operator has matrix
\([g_R]_x\). Its eigenvalues are
\[
\lambda_\pm
=
\frac{A+C\pm\sqrt{(A-C)^2+4B^2}}{2}.
\]
Equivalently, \(\sqrt{\lambda_\pm}\) are the singular values of the weighted
response operator
\[
W^{1/2}J:(T_pP,g^P)\longrightarrow\mathbb R^N
\]
with the standard Euclidean metric on the target. Hence, on the rank-two
locus,
\[
\sigma_{\min}=\sqrt{\lambda_-},
\qquad
\sigma_{\max}=\sqrt{\lambda_+},
\qquad
\kappa_R=\sqrt{\frac{\lambda_+}{\lambda_-}}.
\]
\end{theorem}
\begin{proof}
The identity
\([g_R]_x=J^{\mathsf T}WJ\)
is the coordinate form of
\(g_R=\Theta^*h^{\mathcal Z}.\)
The displayed coefficients are the corresponding Gram coefficients.
For the determinant, apply the standard two-column Cauchy--Binet formula (see, for example, \cite{HornJohnson,GolubVanLoan}) to
\(W^{1/2}J\):
\[
\det(J^{\mathsf T}WJ)
=
\sum_{\alpha<\beta}
\det\!\left((W^{1/2}J)_{\{\alpha,\beta\}}\right)^2.
\]
Because all \(w_\alpha\) are positive, this determinant is positive exactly
when \(J\) has column rank two.
The eigenvalue formula is the characteristic-root formula for a real
symmetric \(2\times2\) matrix. Since the coordinate frame is orthonormal for
\(g^P\), these eigenvalues are precisely the eigenvalues of the relative Gram
operator, and the singular values of the weighted response matrix
\(W^{1/2}J\) are their positive square roots.
\end{proof}
\begin{corollary}[Two scalar response coordinates]
\label{cor:two-scalar-response-coordinates}
If \(N=2\) and
\[
J(p_0)=
\begin{pmatrix}
a&b\\
c&d
\end{pmatrix},
\]
then
\[
[g_R]_{x,p_0}
=
\begin{pmatrix}
w_1a^2+w_2c^2&w_1ab+w_2cd\\
w_1ab+w_2cd&w_1b^2+w_2d^2
\end{pmatrix}
\]
and
\(\det[g_R]_{x,p_0} = w_1w_2(ad-bc)^2.\)
Hence the two scalar observations define local coordinates at \(p_0\)
precisely when
\(ad-bc\neq0.\)
\end{corollary}
%%%%%%%%%%%%%%%%%%%%%%%%%%%%%%%%%%%%%%%%%%%%%%%%%%%%%%%%%%%%%%%%%%%%%%%%%%%%%%%%
\subsection{A four-channel interval illustration}
\label{subsec:four-channel-interval}
%%%%%%%%%%%%%%%%%%%%%%%%%%%%%%%%%%%%%%%%%%%%%%%%%%%%%%%%%%%%%%%%%%%%%%%%%%%%%%%%
The next statement is deliberately finite dimensional. It concerns only the
displayed interval matrix and does not claim that these intervals constitute
a validated enclosure of an exact Jacobian arising from an Einstein
boundary-value problem.
\begin{proposition}[Four-channel interval response calculation]
\label{prop:certified-four-channel-response-metric}
Let \(\mathbf J\) be the interval matrix with centre
\[
J_{\mathrm c}=
\begin{pmatrix}
-0.30787068218 & -0.18575144897\\
 \phantom{-}0.50424122365 & \phantom{-}0.30423065863\\
-0.14093625628 & -0.084561820768\\
 \phantom{-}0.69791494720 & \phantom{-}0.41874838795
\end{pmatrix}
\]
and radius matrix
\[
J_{\mathrm{rad}}=
\begin{pmatrix}
3.45\mathbin{\cdot}10^{-7} & 2.07\mathbin{\cdot}10^{-7}\\
8.29\mathbin{\cdot}10^{-7} & 5.00\mathbin{\cdot}10^{-7}\\
3.82\mathbin{\cdot}10^{-7} & 2.31\mathbin{\cdot}10^{-7}\\
1.03\mathbin{\cdot}10^{-6} & 6.18\mathbin{\cdot}10^{-7}
\end{pmatrix}.
\]
Here \(J\in\mathbf J\) means
\[
|J_{\alpha j}-(J_{\mathrm c})_{\alpha j}|
\le
(J_{\mathrm{rad}})_{\alpha j}
\]
entrywise.
For every such real matrix set
\(G=J^{\mathsf T}J.\)
Then
\[
G\in
\begin{pmatrix}
[0.8559892765,\,0.8559944643]
&
[0.5147600907,\,0.5147632101]
\\
[0.5147600907,\,0.5147632101]
&
[0.3095598706,\,0.3095617462]
\end{pmatrix},
\]
and a cancellation-safe Cauchy--Binet evaluation gives
\(\det G \in [1.9743508816,\,1.9829741725] \mathbin{\cdot}10^{-6}.\)
Consequently every \(J\in\mathbf J\) has column rank two, and
\[
\lambda_{\min}(G)
\ge
1.6939130553\mathbin{\cdot}10^{-6},
\qquad
\lambda_{\max}(G)
\le
1.1655562105.
\]
Equivalently,
\[
\sigma_{\min}(J)
\ge
1.3015041511\mathbin{\cdot}10^{-3}
>
5\mathbin{\cdot}10^{-4},
\]
and
\(\kappa(J)\le829.509.\)
\end{proposition}
\begin{proof}
The enclosures are obtained by outward-rounded RealBall arithmetic at
\(256\)-bit precision.
Direct interval evaluation of
\(G_{11}G_{22}-G_{12}^2\)
suffers from cancellation. Instead, one uses the algebraically identical
Cauchy--Binet expression
\[
\det(J^{\mathsf T}J)
=
\sum_{1\le\alpha<\beta\le4}
\left(
J_{\alpha1}J_{\beta2}
-
J_{\alpha2}J_{\beta1}
\right)^2.
\]
This yields the displayed strictly positive determinant enclosure.
Since \(G\) is positive semidefinite and
\[
\lambda_{\max}(G)\le\operatorname{tr}G,
\]
one has
\[
\lambda_{\min}(G)
=
\frac{\det G}{\lambda_{\max}(G)}
\ge
\frac{\inf\det G}{\sup\operatorname{tr}G}.
\]
The singular-value lower bound follows by taking square roots. For the
condition number, use
\[
\kappa(J)^2
=
\frac{\lambda_{\max}(G)}{\lambda_{\min}(G)}
\le
\frac{(\sup\operatorname{tr}G)^2}{\inf\det G},
\]
which gives the stated numerical upper bound.
\end{proof}
\begin{remark}[Scope of the interval calculation]
Proposition~\ref{prop:certified-four-channel-response-metric} is a rigorous
statement about every real matrix contained in the prescribed interval
matrix \(\mathbf J\). The accompanying Sage program checks this
finite-dimensional assertion by outward-rounded interval arithmetic.

The proposition does not establish that \(\mathbf J\) encloses an exact
response Jacobian produced by an Einstein ODE computation. Such an upstream
validation would require independent control of the ODE propagation,
event-location, differentiation and observation errors.

If a numerical calculation associates \(\mathbf J\) with two shooting
coordinates, the resulting tensor measures observability on that
shooting-coordinate space. It must not be interpreted as a metric on an
Einstein moduli space without an additional restriction argument.
\end{remark}

%%%%%%%%%%%%%%%%%%%%%%%%%%%%%%%%%%%%%%%%%%%%%%%%%%%%%%%%%%%%%%%%%%%%%%%%%%%%%%%%
%%%%%%%%%%%%%%%%%%%%%%%%%%%%%%%%%%%%%%%%%%%%%%%%%%%%%%%%%%%%%%%%%%%%%%%%%%%%%%%%

The preceding examples emphasize algebraic verification and conditional
cohomogeneity-one applications.  The final section turns to a complementary
exact model in which the parameter space, harmonic observations, response
metric, reconstruction map, and conditioning can all be computed globally.

\section{An Explicit Model of Response Geometry: Marked Flat Tori}
\label{sec:flat-tori-response}
%%%%%%%%%%%%%%%%%%%%%%%%%%%%%%%%%%%%%%%%%%%%%%%%%%%%%%%%%%%%%%%%%%%%%%%%%%%%%%%%

We conclude with an exact model in which the response geometry is completely
explicit. The symmetric-space description of marked unit-volume flat two-tori
is classical; see, for example, \cite{Wolf}.  We use this classical family of
normalized Ricci-flat Einstein metrics to exhibit a particular three-channel
response system in closed form: three elementary circle-valued harmonic probes
reconstruct the marked metric. The example illustrates both
complete response and a point that is less visible in the abstract theory:
injectivity of the response operator need not imply uniform conditioning.

\subsection{The marked family and harmonic observations}

Let \(T^2=\mathbb R^2/\mathbb Z^2\) with its standard marking. For
\(G\in\operatorname{Sym}_2^+(\mathbb R)\), let \(g_G\) be the constant metric
with matrix \(G\). It is flat, hence Ricci-flat. Since
\(\operatorname{Vol}(T^2,g_G)=\sqrt{\det G}\), we remove the homothetic freedom
by imposing unit volume and set
\[
 \mathcal S=\{G\in\operatorname{Sym}_2^+(\mathbb R):\det G=1\}
 \cong SL(2,\mathbb R)/SO(2).
\]
This is a smooth two-dimensional manifold with
$$T_G\mathcal S=\{H\in\operatorname{Sym}_2(\mathbb R):
\operatorname{tr}(G^{-1}H)=0\}$$. 
We work on the marked space \(\mathcal S\),
not on its \(SL(2,\mathbb Z)\)-quotient.

\smallskip

For \(m=(m_1,m_2)\in\mathbb Z^2\), define
\(u_m(x)=m_1x^1+m_2x^2\pmod{\mathbb Z}\), with target
\(S^1=\mathbb R/\mathbb Z\) carrying its flat metric. Since \(du_m\) is
parallel, \(u_m:(T^2,g_G)\to S^1\) is harmonic for every \(G\in\mathcal S\).
Its energy is
\[
 E_m(G)=\frac12\int_{T^2}|du_m|_{g_G}^2\,dV_{g_G}
       =\frac12\,m^TG^{-1}m.
\]
Choose \(m_1=(1,0)\), \(m_2=(0,1)\), and \(m_3=(1,1)\), and write
\[
 \Theta(G)=\bigl(E_{(1,0)}(G),E_{(0,1)}(G),E_{(1,1)}(G)\bigr)\in\mathbb R^3,
 \qquad R_G=D\Theta_G.
\]
Thus the observation bundle is the trivial Euclidean bundle
\(\mathcal S\times\mathbb R^3\), and \(g_R=\Theta^*g_{\mathbb R^3}\).

\begin{Thm}[Exact reconstruction and complete response]
\label{thm:flat-torus-reconstruction}
The map \(\Theta:\mathcal S\to\mathbb R^3\) is injective and
\(D\Theta_G\) is injective for every \(G\in\mathcal S\). In particular,
\(g_R\) is a Riemannian metric on all of \(\mathcal S\).
\end{Thm}

\begin{proof}
Writing
\(G^{-1}=\left(\begin{smallmatrix}A&B\\ B&C\end{smallmatrix}\right)\), the
three energies give
\[
 A=2E_{(1,0)},\qquad C=2E_{(0,1)},\qquad
 B=E_{(1,1)}-E_{(1,0)}-E_{(0,1)}.
\]
Hence \(\Theta\) reconstructs \(G^{-1}\), and therefore \(G\), uniquely.
Moreover,
\[
DE_m(G)[H]=-\frac12m^TG^{-1}HG^{-1}m.
\]
If all three derivatives vanish and
\(K=G^{-1}HG^{-1}=(k_{ij})\), the first two probes give
\(k_{11}=k_{22}=0\), while the third gives \(2k_{12}=0\). Thus \(H=0\).
\end{proof}

The theorem separates two notions used throughout the paper: injectivity of
\(D\Theta\) gives first-order observability, whereas injectivity of \(\Theta\)
gives the stronger global reconstruction statement.

\subsection{Explicit response metric and conditioning}

Use upper-half-plane coordinates \(\tau=x+iy\), \(y>0\), with
\[
 G_\tau=\frac1y
 \begin{pmatrix}1&x\\ x&x^2+y^2\end{pmatrix},
 \qquad
 G_\tau^{-1}=\frac1y
 \begin{pmatrix}x^2+y^2&-x\\ -x&1\end{pmatrix}.
\]
Then
\[
 \Theta(x,y)=\frac1{2y}
 \bigl(x^2+y^2,\ 1,\ (x-1)^2+y^2\bigr).
\]
With respect to \((\partial_x,\partial_y)\), the response matrix is
\[
 \mathsf R(x,y)=
 \begin{pmatrix}
 x/y & (y^2-x^2)/(2y^2)\\
 0 & -1/(2y^2)\\
 (x-1)/y & \bigl(y^2-(x-1)^2\bigr)/(2y^2)
 \end{pmatrix},
 \qquad \mathsf G_R=\mathsf R^T\mathsf R.
\]
For later use it is enough to record its determinant:
\[
 \det\mathsf G_R
 =\frac{x^2+(x-1)^2+(x^2-x+y^2)^2}{4y^6}>0.
\]
This independently verifies complete response everywhere. The standard
symmetric-space metric on \(\mathcal S\) is
\(g_{\rm sym}=(dx^2+dy^2)/y^2\), and \(g_R\) is not generally a constant
multiple of it because \(g_R\) remembers the selected marked probes. At the
square torus \(\tau=i\),
\(g_R=dx^2+\frac12dy^2\), while at the hexagonal torus
\(\tau=\frac12+\frac{\sqrt3}{2}i\),
\[
 g_R=\frac23(dx^2+dy^2)=\frac12g_{\rm sym}.
\]
Thus the three-channel response is isotropic at the hexagonal torus.

To measure conditioning intrinsically, use \(g_{\rm sym}\) on the domain and
the Euclidean metric on the observation space. At the square torus the
response condition number is \(\sqrt2\), while at the hexagonal torus it is
\(1\). For fixed \(x\) and \(y\to\infty\), the eigenvalues
\(0<\lambda_-\le\lambda_+\) of \(\mathsf G_R\) satisfy
\[
 \lambda_+=\frac12+O(y^{-2}),\qquad
 \lambda_-=\frac1{2y^2}+O(y^{-4}),
 \qquad
 \kappa_R(x,y)=y+O(y^{-1}).
\]
Hence the response remains complete toward the cusp but becomes increasingly
ill-conditioned. This gives an exact example of the distinction between rank
and quantitative observability.

\subsection{Global response geometry}

Set \((p,q,r)=\Theta(G)\). The reconstruction formula and
\(\det G^{-1}=1\) give the single relation
\[
 4pq-(r-p-q)^2=1.
\]
Under the orthogonal change of observation coordinates
\[
 U=\frac{p+q+r}{\sqrt3},\qquad
 V=\frac{p-q}{\sqrt2},\qquad
 W=\frac{p+q-2r}{\sqrt6},
\]
this becomes \(U^2-2V^2-2W^2=1\). Since \(p,q,r>0\), the image lies on the
upper sheet \(U>0\). Conversely, writing \(\rho^2=V^2+W^2\), one has
\(U=\sqrt{1+2\rho^2}\). In each inverse coordinate \(p,q,r\), the radial
linear term has norm at most \(\sqrt{2/3}\,\rho\), whereas
\(U/\sqrt3>\sqrt{2/3}\,\rho\). Hence \(p,q,r>0\) on the entire upper sheet.
The reconstruction formula then gives a symmetric matrix \(G^{-1}\) with
determinant one and first diagonal entry \(2p>0\); hence Sylvester's criterion
makes it positive definite. It is therefore the inverse of a unique
\(G\in\mathcal S\).

\begin{Prop}[Hyperboloid realization]
\label{prop:flat-torus-hyperboloid}
The observation map is a smooth embedding of \(\mathcal S\) onto
\[
 \mathcal H_R=\{(U,V,W):U^2-2V^2-2W^2=1,\ U>0\},
\]
and \(g_R\) is the Euclidean metric induced on \(\mathcal H_R\).
\end{Prop}

\begin{proof}
Surjectivity onto \(\mathcal H_R\) follows from the preceding reconstruction;
the inverse is smooth because it is obtained by linear reconstruction of
\(G^{-1}\) followed by matrix inversion. Since the change
\((p,q,r)\mapsto(U,V,W)\) is orthogonal and
\(g_R=\Theta^*g_{\mathbb R^3}\), the metric is the Euclidean first fundamental
form of the hyperboloid.
\end{proof}

\begin{Bem}[Riemannian completeness]
The hyperboloid \(\mathcal H_R\) is a closed embedded submanifold of
\(\mathbb R^3\) and is therefore complete with its induced Riemannian metric.
Proposition~\ref{prop:flat-torus-hyperboloid} consequently shows that
\((\mathcal S,g_R)\) is complete.

This is Riemannian completeness and should not be confused with
\emph{complete response}, which means that the response operator is injective
at every point. The present example has both properties. Neither property
implies the other in the general framework developed above.
\end{Bem}

Writing \(V=\rho\cos\vartheta\), \(W=\rho\sin\vartheta\) gives
\(U=\sqrt{1+2\rho^2}\) and therefore
\[
 g_R=\frac{1+6\rho^2}{1+2\rho^2}\,d\rho^2+\rho^2d\vartheta^2,
 \qquad
 K_R(\rho)=\frac{4}{(1+6\rho^2)^2}.
\]
Thus the response metric is positively curved, with curvature decreasing from
\(4\) at the apex to \(0\) at infinity. The apex is characterized by
\(p=q=r=1/\sqrt3\) and corresponds precisely to the hexagonal torus, explaining
the isotropy found above.

Finally, the three probes are tied to the chosen marking. Under
\(SL(2,\mathbb Z)\), lattice covectors transform together with the metric, so
the full collection \(\{E_m:m\in\mathbb Z^2\}\) is natural, whereas the
ordered triple used here is not invariant. We therefore make no claim that
this particular response metric descends to the unmarked moduli space.

The example gives a compact exact model of the general hierarchy
\[
 \text{harmonic observations}\Longrightarrow
 \text{response metric}\Longrightarrow
 \text{first-order observability},
\]
with the additional special feature of global reconstruction. It also shows
that complete response, completeness of the induced response metric, and
uniform quantitative conditioning are logically distinct properties.

%%%%%%%%%%%%%%%%%%%%%%%%%%%%%%%%%%%%%%%%%%%%%%%%%%%%%%%%%%%%%%%%%%%%%%%%%%%%%%%%

\end{document}